\documentclass[11pt]{amsart}

\usepackage[all]{xy}
\usepackage{amssymb,array,xspace}
\usepackage{graphicx}
\usepackage{hyperref}

\newcommand*{\C}{{\mathbb C}}
\newcommand*{\Oo}{{\mathcal O}}
\newcommand*{\CP}{{\mathbb {CP}}}
\newcommand*{\R}{{\mathbb R}}
\newcommand*{\Z}{{\mathbb Z}}
\newcommand*{\uPSL}{{\widetilde{\PSL_2(\R)}}}

\DeclareMathOperator{\Aut}{Aut}
\DeclareMathOperator{\Ext}{Ext}
\DeclareMathOperator{\GL}{GL}
\DeclareMathOperator{\Hom}{Hom}

\DeclareMathOperator{\Int}{Int}
\DeclareMathOperator{\Ker}{Ker}
\DeclareMathOperator{\Mat}{Mat}
\DeclareMathOperator{\pr}{pr}
\DeclareMathOperator{\PSL}{PSL}
\DeclareMathOperator{\SL}{SL}
\DeclareMathOperator{\SU}{SU}
\DeclareMathOperator{\Tot}{Tot}

\def\Im{\mathop{\mathrm{Im}}\nolimits}

\let\eps\varepsilon
\let\ph\varphi

\newtheorem{proposition}{Proposition}[section]

\newtheorem{conj}[proposition]{Conjecture}

\newtheorem{lemma}[proposition]{Lemma}

\newtheorem{corollary}[proposition]{Corollary}

\theoremstyle{remark}
\newtheorem{remark}[proposition]{Remark}

\theoremstyle{definition}

\author{Serge Lvovski}

\address{National Research University Higher School of Economics\\
  NRC ``Kurchatov Institute''--SRISA}
\email{lvovski@gmail.com}

\thanks{The study has been funded within the framework of the HSE
  University Basic Research Program (HSE-BR-2025-061) and by NRC
  ``Kurchatov Institute''--SRISA according to the project FNEF-2024-0001}

\title{On holomorphic submersions with biholomorphic fibers}

\begin{document}



\begin{abstract}
We give explicit examples of holomorphic submersions $p\colon Y\to
X$ such that $Y$ and $X$ are complex manifolds, all the fibers
of~$p$ are biholomorphic to the unit disc in the complex plane, but the
mapping~$p$ is not holomorphically locally trivial.

In our examples the complex manifolds~$Y$ are neighborhoods of the
zero section of a line bundle~$L$ on~$X$; the line bundle~$L$ is
not, in general, supposed to be positive or negative.
\end{abstract}

\maketitle

\section{Introduction}

Suppose that $p\colon Y\to X$ is a submersive surjective holomorphic
mapping of complex manifolds such that all the fibers of~$p$ are
biholomorphic to the same complex manifold~$F$.  Is it true that $p$
is holomorphically locally trivial, that is, for any $x\in X$
there exists a neighborhood $V\ni x$ and a biholomorphism $\ph\colon
p^{-1}(V)\to F\times V$ such that
$\pr_V\circ\ph=p|_{p^{-1}(V)}$?

If $p$ is proper, the answer to this question is yes, as was shown by
W.~Fischer and H.~Grauert in~\cite{FischerGrauert}. For non-proper
mappings there exist partial positive results. For example, a theorem
by Nishino (\cite{Nishino}, see also~\cite{Ohsawa,Chirka}) asserts
that if $\dim Y=2$, $Y$ is Stein, $X$ is the unit disc $D\subset\C$,
and all the fibers of~$p$ are biholomorphic to~$\C$, then $p$ is
biholomorphic to the projection $\C\times D\to D$. Andreotti and
Vesentini~\cite{AV} showed that if all the fibers of a holomorphic
surjection $p\colon Y\to X$ are biholomorphic to each other and Stein,
then $p$ is ``almost holomorphically locally trivial'' in the
following sense: if $x\in X$ and $A\subset p^{-1}(x)$ is an open
subset with compact closure, then there exists an open subset
$\mathcal A\subset Y$ such that $\mathcal A\cap p^{-1}(x)=A $ and
$p|_{\mathcal A}\colon \mathcal A\to X$ is holomorphically locally
trivial (Anderotti and Vesentini say that such a mapping $p$ defines
``a locally pseudo-trivial deformation''). For the case ``dimension of
fibers of~$p$ equals~$1$'' this had been proved earlier by
Narasimhan~\cite{Narasimhan}.

If, however, one is interested in the genuine local triviality, then,
in the non-proper case, there exist many counterexamples, that is,
surjective holomorphic submersions $p\colon Y\to X$ such that all the
fibers of~$p$ are biholomorphic but $p$ is not locally trivial.

A quick and silly (counter)example may be produced if one does not
demand that the fibers of~$p$ be trivial. To wit, if $p'\colon Y'\to
X$ is a holomorphic covering with infinite fibers ($Y'$ and $X$ are
supposed to be connected), one may pick a point $a\in Y'$ and put
$Y=Y'\setminus\{a\}$, $p=p'|_Y$.  The mapping $p\colon Y\to X$ is a
surjective submersion and all its fibers, being countable discrete
spaces, are biholomorphic, but $p$ is not locally trivial, even
topologically.

There are examples of not locally trivial holomorphic submersions with
biholomorphic and connected fibers, too; most of them belong to
folklore. For instance, the natural projection $p\colon B\to
D$, where $B$ is the unit ball in~$\C^2$ and $D$ is the unit disc
in~$\C$, is not holomorphically locally trivial (see
Remark~\ref{remark} for one of several possible proofs).

The aim of this paper is to give a family of examples of holomorphic
submersions~$p\colon Y\to X$ such that all the fibers of~$p$ are
biholomorphic to the unit disk but $p$ is not holomorphically locally
trivial. All our examples will be neighborhoods of
the zero section in the total space of some complex line bundle~$L$
on~$X$ (needless to say, the natural projection of any such
neigborhood on~$X$ is a surjective submersion).

If the line bundle~$L$ is positive, one can introduce an Hermitian
metric on it and define the neighborhood in question as the set of
vectors in the fibers such that their norms are less that a
fixed~$\eps>0$; using pseudoconcavity, it is easy to show that the
resulting submersion is not locally trivial (see
Proposition~\ref{positive} and its proof in Section~\ref{sec:proofs}).

All the other examples are of the following form. We indicate
classes of pairs~$(X,L)$, where $X$ is a complex manifold and $L$ is a
line bundle on~$X$, such that if $U$ is a neighborhood of the zero
section of~$L$ in its total space such that the intersection of $U$ with any
fiber is biholomorphic to the disc in the complex plane, then the
projection $U\to X$ is not locally trivial (see
Propositions~\ref{prop:simply-connected}, \ref{Milnor_style},
and~\ref{tori}). In Propositions~\ref{prop:simply-connected}
and~\ref{tori} the bundle~$L$ does not need to be positive or negative.

The main idea of the proofs of these propositions is as follows. We
show that a holomorphically locally trivial bundle with
fiber~$D=\{z\in\C\colon |z|<1\}$ is a flat $\PSL_2(\R)$-bundle
(Corollary~\ref{corollary}), after which we use the results of
Milnor's old paper~\cite{Milnor}.  Precise statements will be given in
Section~\ref{sec:statements} below, after we have fixed notation and
terminology.

The paper is organized as follows. In Section~\ref{sec:statements} we
state the main results. In the following two sections we prove two
lemmas, and in Section~\ref{sec:proofs} we prove what is stated in
Section~\ref{sec:statements}. In the appendix several properties of
flat line bundles are gathered for which (properties) I did not manage
to find suitable references.

\subsection*{Acknowledgements}

I am grateful to Laszlo Lempert for valuable comments on the first
version of this text and for the example from Remark~\ref{remark}; I
am grateful to Stefan Nemirovski for useful discussions.

\subsection{Notation and conventions}

$D=\{z\in\C\colon |z|<1\}$ denotes the unit disc in the comlex plane,
and $H=\{z\in\C\colon \Im z>0\}$ stands for the upper half-plane.

By M\"obius transformations we mean linear fractional transformations
of the Riemann sphere~$\CP^1$.

The group of complex numbers of modulus~$1$ will be denoted
by~$\mathrm U(1)$. 

If $L$ is a line bundle over a manifold~$X$, then $\Tot_L$ is its
total space; if $x\in X$, then $L_x\subset\Tot_L$ is the fiber of~$L$
over~$x$. 

If $L$ is a complex line bundle on a topological space~$X$, then
$c_1(L)\in H^2(X,\Z)$ is its first Chern class and
$c_\R(L)\in H^2(X,\R)$ is the image of~$c_1(L)$ in~$H^2(X,\R)$.

If $p\colon Y\to X$ is a surjective submersion of smooth connected manifolds, then by \emph{relative tangent bundle} corresponding
to~$p$ we mean the vector bundle of rank $\dim Y-\dim X$ on~$Y$ such
that its fiber over a point $y\in Y$ is the tangent space $T_yF_y$,
where the manifold $F_y$ is the fiber of $p$ containing~$y$.

If $G$ is a topological group and $Y\to X$ is a fiber bundle with
structure group~$G$, one says that this bundle is \emph{flat} if it
admits a trivialization over an open cover~$\{V_\alpha\}$ of $X$ such
that the transition functions $g_{\alpha\beta}\colon V_\alpha\cap
V_\beta\to G$ are locally constant.  If $X$ is Hausdorff, path
connected, and locally simply connected, then a $G$-bundle is flat if
and ony if it is induced by a homomorphism $\pi_1(M)\to G$ (see
\cite[\S\,13]{Steenrod}). 

\section{Statements}\label{sec:statements}

We begin with an optimistic conjecture.

\begin{conj}\label{conj}
Suppose that $L$ is a holomorphic line bundle on a connected complex
manifold $X$; let $p\colon\Tot_L\to X$ be the natural projection.
Suppose that $U\subset \Tot_L$ is a neighborhood of the zero section of~$L$
such that, for any $x\in X$, the intersection $L_x\cap U$
is biholomorphic to the unit disc $D\subset\C$.

If $c_\R(L)\ne 0$, then the restriction
$p|_U\colon U\to X$ is never holomorphically locally trivial.
\end{conj}

Observe at once that one cannot drop the hypothesis $c_\R(L)\ne0$ in
this conjecture. 

Indeed, if $L$ is an arbitrary (topological) line bundle on a complex
manifold~$X$ such that~$c_\R(L)=0$, Proposition~\ref{4equiv} implies
that $L$ admits a trivialization over an open cover $X=\bigcup
V_\alpha$ such that the transition functions $g_{\alpha\beta}\colon
V_\alpha\cap V_\beta\to\C^*$ are locally constant of
modulus~$1$. Gluing the complex manifolds $V_\alpha\times \C$ by these
transition functions, one obtains a structure of a complex line bundle
on~$L$.

If $f_\alpha\colon \pi^{-1}(V_\alpha)\to V_\alpha\times\C$ are the
trivializations of the complex line bundle~$L$ to which the transition
functions~$g_{\alpha\beta}$ correspond (here, $\pi$ is the projection
to~$X$ from the total space of our bundle), then, putting
\begin{equation}\label{eq:tubular}
\|v\|=|\pr_2(f_\alpha(v))|\quad \text{for
  $v\in\pi^{-1}(V_\alpha)\subset \Tot_L$},
\end{equation}
one obtains a well-defined Hermitian norm on~$L$; the open set
$U=\{v\in \Tot_L\colon \|v\|<1\}$ is a neigborhood of the zero section
and the projection $\pi|_U\colon U\to X$ is a holomorphically locally
trivial bundle with fiber $D$.

Observe also that for any holomorphic line bundle $L$ (with an
Hermitian metric) on $X$
the open subset $U\subset\Tot_L$ defined by the
formula~\eqref{eq:t} below is such that $U\cap L_x$ is biholomorphic to
$D$ for any~$x$ and the projection $\pi|_U\colon U\to X$ is locally
trivial in the $C^\infty$ sense.

I do not have a proof for Conjecture~\ref{conj} in full generality.
Propositions~\ref{positive}, \ref{prop:simply-connected},
\ref{Milnor_style}, and~\ref{tori}, which are stated below, are
partial results; some of them are not, strictly speaking, consequences
of this conjecture.

First, one has a simple result for positive line bundles.

\begin{proposition}\label{positive}
Suppose that $L$ is a positive holomorphic line bundle on a compact
complex manifold~$X$.  Let $\|\cdot\|$ be an arbitrary Hermitian
metric on~$L$ regarded as a function on~$\Tot_L$, and let $\eps$ be an
arbitrary positive number.

If one puts
\begin{equation}\label{eq:t}
U=\{v\in\Tot_L\colon \|v\|<\eps\},  
\end{equation}
then all the fibers of the natural submersion $p\colon U\to X$ are
biholomorphic to the unit disk but $p$ is not
holomorphically locally trivial. 
\end{proposition}

Then, we show that Conjecture~\ref{conj} is true for simply connected
complex manifolds. Actually, in this case a somewhat stronger result
holds. 

\begin{proposition}\label{prop:simply-connected}
If $X$ is a simply connected complex manifold, $L$ is a holomorphic
line bundle $L$ on $X$ that is holomorphically non-trivial, and
$U\subset \Tot_L$ is a neighborhood of the zero section
such that, for any $x\in X$, the intersection $L_x\cap U$
is biholomorphic to the unit disc $D\subset\C$, then the natural
projection $p\colon U\to X$ is not holomorphically locally trivial.
\end{proposition}

This proposition is not a consequence of Conjecture~\ref{conj} since
there exist pairs $(X,L)$ where $X$ is simply connected and $L$ is a
line bundle that is trivial topologically (so $c_\R(L)=0$, of course)
but not holomorphically (put $X=\C^2\setminus\{0\}$, for example).

For the case of not simply connected~$X$, we have the following two results.

\begin{proposition}\label{Milnor_style}
Conjecture~\ref{conj} is true if $X$ is a compact Riemann surface of
genus~$g\ge 2$ 
and $L$ is a holomorphic line bundle on~$X$ such that ${|\deg L|\ge g}$.   
\end{proposition}

\begin{proposition}\label{tori}
Conjecture~\ref{conj} is true if $X$ is a complex torus \textup(that
is, $X=\C^n/\Lambda$, where $\Lambda\subset\C^n$ is a lattice of
rank~$2n$\textup) or $X\cong(\C^*)^m$ for some~$m\ge2$.
\end{proposition}

It is highly possible that Proposition~\ref{tori} may be extended to
arbitrary complex manifolds with commutative fundamental group, which
would imply our current Proposition~\ref{prop:simply-connected}.

Before passing to proofs, we mention one consequence of
Proposition~\ref{prop:simply-connected}. 

\begin{corollary}
Suppose that $H\subset\CP^n$, $n\ge2$, is a hyperplane and
$a\in\CP^n\setminus H$ is a
point. Let $K\subset\CP^n\setminus H$  be a compact subset such that
$a\in\Int(K)$ and, for
any line $\ell\subset\CP^n$, $\ell\ni a$, the set
$\ell\setminus K$ is
simply connected and not empty.

If one puts $Y=\CP^n\setminus K$ and denotes by $p\colon Y\to H$ the
projection from~$p$, then all the fibers of~$p$  are biholomorphic
to~$D$ but $p$ is not holomorphically locally trivial.
\end{corollary}

\begin{proof}
The fibers of~$p$ are biholomorphic to~$D$ by Riemann's mapping theorem.
  
Observe that $\CP^n\setminus\{a\}$ is biholomorphic to the total space of
the line bundle $\Oo_{\CP^{n-1}}(1)=\Oo_H(1)$ and that, with this
identification in mind, the canonical projection from the total space
to the base is just the projection onto~$H$ with center~$a$. Since
$H$ is simply connected and $c_\R(\Oo_{\CP^{n-1}}(1))\ne 0$, the
assertion follows from Proposition~\ref{prop:simply-connected}.
\end{proof}

\section{A lemma on Euler class}

Recall some standard facts.

Suppose that $X$ is a good enough topological space, say, a
CW-complex.  To any $\GL_2^+(\R)$-principal bundle~$\xi$ on~$X$, where
$\GL_2^+(\R)$ is the group of real $2\times2$-matrices with positive
determinant, one can assign its \emph{Euler class} $e(\xi)\in
H^2(X,\Z)$; it coincides with the Euler class of the oriented real
vector bundle of rank~$2$ associated to~$\xi$. If one is given a Lie
group~$G$ together with a homomorphism $G\to\GL_2^+(\R)$, then any
$\GL_2^+(\R)$-bundle can be regarded as a $G$-bundle, so it makes
sense to speak of Euler classes of $G$-bundles in this situation. In
particular, having in mind the standard embeddings
$\SL_2(\R)\hookrightarrow \GL_2^+(\R)$ and $\C^*\hookrightarrow
\GL_2^+(\R)$, the latter being $u+iv\mapsto \left(
\begin{smallmatrix}u&-v\\ v&u  \end{smallmatrix}\right)$,
one can speak of the Euler class of $\SL_2(\R)$-bundles or
of complex line bundles (which are $\C^*$-bundles).

Suppose now that $X$ is a smooth manifold, $F$ is a complex manifold, and
$G$ is a real Lie group acting on~$F$ by biholomorphisms. If $p\colon
E\to X$ is a $G$-bundle over~$X$ with fiber~$F$, then the
relative tangent bundle~$T_{E|X}$ is a complex
vector bundle of (complex) rank~$\dim_\C F$.

Now we identify $\SL_2(\R)$ with the
group
\begin{equation*}
  \SU(1,1)=\left\{
  \begin{pmatrix}
    p& q\\ \bar q&\bar p
  \end{pmatrix}
\in \Mat_2(\C)\colon |p|^2-|q|^2=1
  \right\}.
\end{equation*}
The latter acts on the unit disc~$D$ by
biholomorphisms
$z\mapsto \frac{pz+q}{\bar q z+\bar p}$,
each automorphism of~$D$ is of this form, and
$\Aut(D)=\SL_2(\R)/\{\pm1\}=\SU(1,1)/\{\pm1\}$. 

\begin{lemma}\label{ghnm}
Suppose that $X$ is a smooth manifold, $\xi$ is a principal
$\SL_2(\R)$-bundle over~$X$, and $p\colon U\to X$ is the bundle
over~$X$ with fiber~$D$ 
corresponding to~$\xi$ and the action of $\SL_2(\R)=\SU(1,1)$
described above.

If $s\colon X\to U$ is a smooth section, then
$e(\xi)=s^*c_1(T_{U|X})$. 
\end{lemma}

\begin{proof}
Suppose that the bundle $\xi$ is trivializable over an open cover
$X=\bigcup V_\alpha$, with transition functions $g_{\alpha\beta}\colon
V_\alpha\cap V_\beta\to\SU(1,1)$; let $f_\alpha\colon
p^{-1}(V_\alpha)\to V_\alpha\times D$ be the corresponding
trivializations of $p\colon U\to X$. One has
$f_\alpha(y)=(p(y),h_\alpha(y))$, where $h_\alpha\colon
p^{-1}(V_\alpha)\to D$. Put $s_\alpha=h_\alpha\circ
s|_{V_\alpha}\colon V_\alpha\to D$.  Let us modify the mappings
$h_\alpha$ by replacing them with
\begin{equation*}
  \tilde h_\alpha\colon y\mapsto 
  \frac{h_\alpha(y)-s_\alpha(p(y))}
       {-\overline{s_\alpha(p(y))}h_\alpha(y)+1}.
\end{equation*}
(Observe that even if $X$ is a complex manifold and both the
bundle~$p$ and the section~$s$ are holomorphic, as well as the
trivializations~$f_\alpha$, the new trivializations $\tilde
f_\alpha=(p,\tilde h_\alpha)$ need not be holomorphic.)

With these new trivializations, the functions $\tilde s_\alpha=\tilde
h_\alpha\circ s|_{V_\alpha}\colon V_\alpha\to D$ are identically zero,
so the new transition finctions $\tilde g_{\alpha\beta}\colon
V_\alpha\cap V_\beta\to \SU(1,1)$ map the points of $V_\alpha\cap
V_\beta$ to zero-preserving automorphisms of~$D$. Hence
(maybe after passing to a
finer cover) these transition functions will have the form
\begin{equation*}
\tilde g_{\alpha\beta}(x)=
\begin{pmatrix}
  e^{\pi it_{\alpha\beta}(x)} & 0 \\
  0 & e^{-\pi it_{\alpha\beta}(x)}
\end{pmatrix},
\end{equation*}
where $t_{\alpha\beta}(x)\in\R$.

Now it is clear that both $e(\xi)$ and $c_1(s^*T_{U|X})$ may be
represented by the same \v{C}ech
cocycle~$\{t_{\alpha\beta}+t_{\beta\gamma}-t_{\alpha\gamma}\}$. Since,
for complex line bundles, Euler class and first Chern class coincide,
we are done.
\end{proof}

\section{An observation on the universal covering of
  $\PSL_2(\R)$}

Let $\uPSL$ be the universal covering of $\PSL_2(\R)$ (we regard
$\uPSL$ as an abstract group, ignoring the topology).  The aim of this
section is to give an explicit description of this group in the form
suitable for our purposes. (I believe that this description is known,
but I failed to find an approriate reference.) A more algebraic
description can be found, for example, in \cite[Section 6]{Rawnsley}.

We fix notation first. If $\gamma_1,\gamma_2\colon [0;1]\to X$ are two
paths in a topological space~$X$, and if $\gamma_1(1)=\gamma_2(0)$,
then $\gamma_1\cdot\gamma_2\colon [0;1]\to X$ is the path defined by
the standard formula
\begin{equation*}
\gamma_1\cdot\gamma_2(t)=
\begin{cases}
  \gamma_1(2t),&0\le t\le\frac12,\\
  \gamma_2(2t-1),&\frac12\le t\le1;
\end{cases}
\end{equation*}
if a group $G$ acts on~$X$ and $\gamma$ is a path in~$X$, then
$g\gamma\colon [0;1]\to X$, where $g\in G$, is the path defined by the
formula $(g\gamma)(t)=g(\gamma(t))$.

We will identify the group $\PSL_2(\R)$ with the group of
automorphisms of the upper half-plane~$H$ or of the unit disc~$D$.

In the first case, we will consider the action of $\PSL_2(\R)=\Aut(H)$
on $\R\cup\{\infty\}$, in the second case, we will consider the action
of $\PSL_2(\R)=\Aut(D)$ on~$\partial D$ (the boundary of~$D$).

Observe that $\uPSL$ is the universal covering of $\SL_2(\R)$ as well.

\begin{proposition}\label{description}
As an abstract group, $\uPSL$ is isomorphic to the set of pairs
$(g,[\gamma])$, where $g\in\Aut(H)$ and   $[\gamma]$ is a homotopy
class of paths $\gamma\colon [0;1]\to\R\cup\{\infty\}$ for which
$\gamma(0)=0$, $\gamma(1)=g(0)$, with multiplication defined by the
formula
\begin{equation}\label{eq:mult}
  (g_1,[\gamma_1])(g_2,[\gamma_2])=
(g_1g_2,[\gamma_1\cdot g_1\gamma_2]).
\end{equation}

Alternatively, $\uPSL$ is isomorphic to the set of pairs
$(g,[\gamma])$, where $g\in\Aut(D)$ and   $[\gamma]$ is a homotopy
class of paths $\gamma\colon [0;1]\to\partial D$ for which
$\gamma(0)=1$, $\gamma(1)=g(1)$, with multiplication defined by the
same formula~\eqref{eq:mult}.
\end{proposition}

\begin{proof}
If $G$ is an arbitrary topological group the
underlying space of which is Hausdorff, path connected, and locally simply
connected, then its universal covering $\tilde G$ is the set of pairs
$(g,[\gamma])$, where $[\gamma]$ is a homotopy class of paths in $G$
joining the neutral element~$e$ and the point~$g$, and the
multiplication law is
$(g_1,[\gamma_1])(g_2,[\gamma_2])=(g_1g_2,[\gamma_1\gamma_2])$,  
where $(\gamma_1\gamma_2)(t)=\gamma_1(t)\gamma_2(t)$. It is easy to
see that $\gamma_1\gamma_2\sim \gamma_1\cdot g_1\gamma_2$, so the
multiplication in~$\tilde G$ may be as well described by the
formula~\eqref{eq:mult}, in which $\gamma_j$ are understood as paths
in~$G$ starting from~$e$.

Since the mapping $g\mapsto g(0)$ (resp.\ $g\mapsto g(1)$) is a
homotopy equivalence of $\Aut(H)$ and~$\R\cup\{\infty\}$ (resp.\ of
$\Aut(D)$ and~$\partial D$), the proposition follows.
\end{proof}

\section{Proofs}\label{sec:proofs}

\begin{proof}[Proof of Proposition~\ref{positive}]
Arguing by contradiction, suppose that the projection~$p\colon U\to
X$, where $U$ and $X$ are as in~\eqref{eq:t}, is holomorphically
locally trivial. For a point $x\in X$, let $V\ni x$ be a neighborhood
over which $p$ is trivial and such that $V$ is biholomorphic to a
Stein open subset in~$\C^n$. Then $p^{-1}(V)\cong V\times D$ is also
Stein. As Grauert's argument from the proof of~\cite[\S\,3,
  Satz~1]{Grauert62} (with signs reversed) shows, positivity of~$L$
implies that $p^{-1}(V)$ is biholomorphic to an open subset
of~$\C^{n+1}$ containing a strongly pseudoconcave part of boundary.
Now the Levi extension theorem implies that this open subset cannot be
a domain of holomorphy (see~\cite{Siu}). Since, on the other hand,
this open subset is biholomorphic to~$V\times D$, we arrived at a
contradiction.
\end{proof}

To proceed, we need a lemma.

\begin{lemma}\label{lemma}
Suppose that $V\subset\C^n$ is an open subset and the mapping $g\colon
V\to\Aut(D)$, where $D\subset\C$ is the unit disc, is such that 
\begin{equation}\label{eq:G}
\begin{aligned}
  G\colon &V\times D\to V\times D,\\
  &(z,w)\mapsto(z,g(z)w)  
\end{aligned}
\end{equation}
is a biholomorphism from $V\times D$ to
itself. Then $g$ is locally constant.
\end{lemma}

\begin{proof}
For any $z\in V$ put $s(z)=\pr_2(G^{-1}(z,0))$; since $G$ is holomorphic,
the mapping $z\mapsto s(z)$ is a holomorphic mapping from $V$
to~$D$. Hence, $V$ can be covered by open subspaces~$V_\alpha$ such
that, for each~$\alpha$,
\begin{equation}\label{eq:lemma}
  G|_{V_\alpha\times D}\colon (z,w)\mapsto e^{i\theta(z)}
  \frac{w-s(z)}{1-\overline{s(z)}w},  
\end{equation}
where $\theta(z)\in\R$; without loss of generality we may and will
assume that $V=V_\alpha$
for some~$\alpha$ and that
$V$ is open and connected. If $z\in V\subset\C^n$, put $z=(z_1,\dots,z_n)$.

Since the right-hand side of~\eqref{eq:lemma} is holomorphic as a
function of~$z$, on has, for any fixed~$w\in D$ and for any~$j$, $1\le
j\le n$,
\begin{multline*}
  0\equiv\frac\partial{\partial \bar z_j}
  \left(e^{i\theta(z)}\frac{w-s(z)}{1-\overline{s(z)}w}\right)\\
{}=e^{i\theta(z)}\left(
\frac{w-s(z)}{1-\overline{s(z)}w}\cdot
i\frac{\partial\theta(z)}{\partial \bar z_j}+
\frac{w-s(z)}{(1-\overline{ s(z)}w)^2}\cdot
\frac{\partial\overline{s(z)}}{\partial \bar z_j}\right)
\end{multline*}
for any $w\in D$ and any $j$, whence
\begin{equation*}
  i(1-\overline{s(z)}w)\frac{\partial\theta(z)}{\partial \bar z_j}
  + \overline{\left(\frac{\partial s}{\partial z_j}\right)}
  =0\quad\text{for
any $w\in D$ and any $j$,}  
\end{equation*}
or, equivalently,
\begin{align}
  i\frac{\partial\theta(z)}{\partial \bar
    z_j}+\overline{\left(\frac{\partial s}{\partial z_j}\right)}
  &=0,\label{eqa}\\ 
\overline{s(z)}\cdot \frac{\partial\theta(z)}{\partial \bar
    z_j}&=0\label{eqb}  
\end{align}
for any~$j$, $1\le j\le n$.

If the function~$s$ is not identically zero, then it follows
from~\eqref{eqb} that $\theta$ is holomorphic; being real valued, it
must be constant. Now~\eqref{eqa} implies that $s$ is constant, too,
and we are done.

If the function~$s$ is identically zero, then \eqref{eqa} implies that
the real valued function~$\theta$ is holomorphic, hence 
constant, and we are done.
\end{proof}

\begin{corollary}\label{corollary}
Suppose that $X$ is a connected complex manifold and~$p\colon Y\to
X$ is a locally trivial holomorphic bundle with fiber~$D$. Then this
bundle is flat.
\end{corollary}

\begin{proof}
It follows from Lemma~\ref{lemma} that the transition functions of $p$
are locally constant, whence the result.  
\end{proof}

\begin{remark}\label{remark}
This corollary implies that if $B=\{(z,w)\in\C^2\colon
|z|^2+|w|^2<1\}$, $D=\{z\in\C\colon |z|<1\}$, and $p\colon
(z,w)\mapsto z$, then the holomorphic submersion~$p\colon B\to D$ is
not holomorphically locally trivial.  By contradiction, if $p$ were
locally trivial, then, since $D$ is simply connected, it would be
globally trivial as well, whence $B\cong D\times D$, a contradiction.
\end{remark}

\begin{proof}[Proof of Proposition~\ref{prop:simply-connected}]
Arguing by contradiction, suppose that $p|_U\colon U\to X$ is holomorphically
locally trivial. Then, by virtue of Corollary~\ref{corollary},
the bundle~$p|_U$ is induced
by a homomorphism $\pi_1(X)\to\PSL_2(\R)$.
Since $X$ is simply connected, the bundle $U\to X$ is
trivial, that is, that there exists a biholomorphism $F\colon U\to X\times
D$ such that $\pr_X\circ F=p$. Hence, if $T_{U|X}$ is the relative
tangent bundle of the bundle~$p$, one has
\begin{equation*}
T_{U|X}=F^*T_{X\times D|X}=F^*\pr_D^*T_D,  
\end{equation*}
so the bundle $T_{U|X}$ is trivial. It follows that if $Z\subset
U\subset \Tot_L$ is the zero section, then $T_{U|X}|_Z$ is the trivial
bundle as well.

On the other hand, it is clear that, under the obvious identification
of~$X$ and~$Z$, $T_{U|X}|_Z\cong T_{\Tot_L|X}|_Z\cong L$.
Since $L$ is not trivial, we arrived at a
contradiction. 
\end{proof}

\begin{proof}[Proof of Proposition~\ref{Milnor_style}]
Arguing by contradiction, if $p\colon U\to X$ is holomorphically
locally trivial, then, by virtue of Corollary~\ref{corollary}, $p$
is a flat bundle with $\PSL_2(\R)$ as structure group.

Now let $f\colon \tilde X\to X$ be an (unramified) covering of
degree~$2$; then $\tilde X$ is a compact Riemann surface of
genus~$2g-1$. If one puts $\tilde L=f^*L$ and $\tilde
U=(\Tot_f)^{-1}(U)$, where $\Tot_f\colon \Tot_{\tilde L}\to\Tot_L$ is
induced by~$f$, 
then $\tilde L$ is a line bundle on~$\tilde X$ of degree $2\deg L$,
$\tilde U\subset\Tot_{\tilde L}$ is a neighborhood of the zero
section of~$\tilde L$, and $\tilde U\to\tilde X$ is a flat
$\PSL_2(\R)$-bundle. Let~$\xi$
be the corresponding principal $\PSL_2(\R)$-bundle.

Observe that the structure group of the latter bundle can be reduced
(lifted) to $\SL_2(\R)$. Indeed, if one is given a
$\PSL_2(\R)$-bundle~$E$ over a topological space~$X$ and if $\xi\in
\check H^1(X,\PSL_2(\R))$ (non-commutative \v{C}ech cohomology) is the
class of a cocycle defining this bundle, then the only obstruction to
the structure group being liftable to $\SL_2(\R)$ is the class
$\delta(\xi)\in H^2(X,\Z/2)$, where the mapping $\delta$ is induced by
the exact sequence
\begin{equation*}
0\to\Z/2\Z\to\SL_2(\R)\to\PSL_2(\R)\to 1;  
\end{equation*}
if $f\colon \tilde X\to X$ is a continuous mapping, then
$\delta(f^*\xi)=f^*(\delta(\xi))$ 
(see for example \cite[Section 1]{NWW}).
Since $\deg f=2$ in our case, one has
$\delta(f^*\xi)=f^*\delta(\xi)=0$, so $\tilde U$ can be regarded as a
flat $\SL_2(\R)$-bundle. According to Proposition~\ref{ghnm}, the
Euler class of this bundle equals $s^*c_1(T_{\tilde U|\tilde X})$, where
$s\colon \tilde X\to \tilde U\subset\Tot_{\tilde L}$ is the zero
section. It is clear that 
\begin{equation}\label{eq:rel.tan}
\tilde s^*  T_{\tilde U|\tilde X}\cong \tilde s^*  T_{\Tot_{\tilde
    L}|\tilde X}\cong \tilde L,
\end{equation}
so the Euler class of our $\SL_2(\R)$-bundle in~$\tilde X$
equals~$c_1(\tilde L)$ and its degree (value on the fundamental class
in $H_2(\tilde X,\Z)$, speaking formally) equals $2\deg L$. Now using
Milnor's Theorem~1 from the paper~\cite{Milnor} one concludes that
this is impossible if
\begin{equation*}
2|\deg L|=|\deg \tilde L|\ge g(\tilde X)=2g-1
\Leftrightarrow |\deg L|\ge g,
\end{equation*}
whence the result.
\end{proof}

It remains to prove Proposition~\ref{tori}. To that end, we need another
lemma.

\begin{lemma}\label{commutators}
Let $\pi\colon \uPSL\to\PSL_2(\R)$ be the universal covering
of~$\PSL_2(\R)$; suppose that $g_1,g_2\in\PSL_2(\R)$ and that
$\Gamma_1,\Gamma_2\in\uPSL$ are such that $\pi(\Gamma_i)=g_i$.  If
$g_1$ and $g_2$ commute, then $\Gamma_1$ and $\Gamma_2$ also commute.
\end{lemma}

\begin{proof}
We may and will assume that neither $g_1$ nor~$g_2$ is identity.
Moreover, since $\Ker p$ lies in the center of $\uPSL$, it suffices,
for given $g_1$ and $g_2$, to check the assertion for just one pair
$\Gamma_1,\Gamma_2$.
  
Let us regard $g_1$ and $g_2$ as M\"obius transformations mapping $H$
into itself. We claim that $g_1$ and $g_2$ are either both hyperbolic,
or both elliptic, or both parabolic, and their sets of fixed points
are the same.

Indeed, Theorem~2 from~\cite[Chapter~I, 9F]{Lehner} asserts that if
$g_1$ and~$g_2$ are arbitrary commuting non-identity M\"obius
transformations, then either their sets of fixed points coincide or
both $g_1$ and~$g_2$ are elliptic M\"obius transformations of
order~$2$.  In the former case our assertion follows immediately. In
the latter case one may assume without loss of generality that $g_1$
is of the form~$z\mapsto-1/z$. The fixed points of~$g_1$ are~$i$
and~$-i$, and $g_2$, which commutes with~$g_1$, must either preserve
them separately or interchange them. In the first case $g_2=g_1$
or~$g_2=\mathrm{id}$, in the second case $g_2(H)\not\subset H$, a
contradiction.

Identifying $\PSL_2(\R)$ with a $\Aut(H)$ or $\Aut(D)$ and
conjugatinng by an appropriate element of $\uPSL$, we see in view of the
above that it suffices to consider the following three cases.

\begin{enumerate}
\item\label{i:h} The hyperbolic case: $\PSL_2(\R)$ is identified with
  $\Aut(H)$, $g_1\colon z\mapsto \lambda z$, $g_2\colon z\mapsto \mu
  z$, $\lambda, \mu\in\R_+$.

\item\label{i:p} The parabolic case: $\PSL_2(\R)$ is identified with
  $\Aut(H)$, $g_1\colon z\mapsto z+a$, $g_2\colon z\mapsto z+b$,
  $a,b\in\R$.

\item\label{i:e} The elliptic case: $\PSL_2(\R)$ is identified with
  $\Aut(D)$, $g_1\colon z\mapsto e^{i\ph} z$, $g_2\colon z\mapsto
  e^{i\theta} z$, $0\le \ph,\theta<2\pi$.
\end{enumerate}
Now we use the description of $\uPSL$ from
Proposition~\ref{description}. In case~\eqref{i:h}, one may put
$\Gamma_1=(g_1,[\gamma])$, $\Gamma_2=(g_2,[\gamma])$, where $\gamma\colon
[0;1]\to\R\cup\{\infty\}$, $\gamma(t)\equiv0$;
in case~\eqref{i:p},
put $\Gamma_1=(g_1,[\gamma_1])$, $\Gamma_2=(g_2,[\gamma_2])$, where
$\gamma_1(t)=ta$, $\gamma_2(t)=tb$;
in case \eqref{i:e},
put $\Gamma_1=(g_1,[\gamma_1])$, $\Gamma_2=(g_2,[\gamma_2])$, where
$\gamma_1(t)=e^{it\ph}$, $\gamma_2(t)=e^{it\theta}$. 

Formula~\eqref{eq:mult} shows that in each of these cases $\Gamma_1$
and $\Gamma_2$ commute.
\end{proof}

\begin{proof}[Proof of Proposition~\ref{tori}]
We will give a proof for the case in which $X\cong\C^n/\Lambda$ is a
compact complex torus; the proof for $X\cong(\C^*)^n$ is
similar, with obvious alterations.

Arguing by contradiction, suppose that $p\colon U\to X$ is
holomorphically locally trivial. Then, by Corollary~\ref{corollary},
$p$ is a flat bundle with $\PSL_2(\R)$ as structure group. Let
$f\colon X\to X$ be the mapping $z\bmod\Lambda\mapsto 2z\bmod\Lambda$;
as in the proof of Proposition~\ref{Milnor_style}, we put $\tilde
L=f^*L$ and $\tilde U=\Tot_f^{-1}(U)$. Again, $\tilde U\to X$ is a
flat $\PSL_2(\R)$-bundle.

It is clear that the induced mapping $f^*\colon H^2(X,\Z)\to
H^2(X,\Z)$ is multiplication by~$4$ and the induced mapping
$f^*_2\colon H^2(X,\Z/2\Z)\to H^2(X,\Z/2\Z)$ is zero. The latter
implies, again as in the proof of Proposition~\ref{Milnor_style}, that
the structure group of the $\PSL_2(\R)$-bundle $\tilde U\to X$ can be
lifted to~$\SL_2(\R)$ and (using Proposition~\ref{ghnm} and the
isomorphisms~\eqref{eq:rel.tan} with $X$ substitued for~$\tilde X$)
that the Euler class of the resulting flat $\SL_2(\R)$-bundle equals
$c_1(\tilde L)=f^* c_1(L)$.

Observe now that, since $X$ is topologically a torus, the hypothesis
$c_\R(L)\ne 0$ implies $c_1(L)\ne 0$; hence, $c_1(f^*L)=4c_1(L)\ne 0$,
so the Euler class of the flat $\SL_2(\R)$-bundle $\tilde U$ on~$X$ is
not zero. This, however, contradicts the following proposition.
\end{proof}

\begin{proposition}\label{on_torus}
If $X$ is a real torus \textup(i.\,e., a finite product of
circles\textup) and $\xi$ is a flat $\SL_2(\R)$-bundle on~$X$, then the
Euler class of~$\xi$ is zero.
\end{proposition}

\begin{proof}
Suppose that $X=\R^n/\Lambda$, where $\Lambda$ is a lattice. If
$(v_1,\dots,v_n)$ is a basis of~$\Lambda$, then
$H_2(X,\Z)\cong\bigwedge^2\Lambda$ is a free $\Z$-module with
basis~$(v_i\wedge v_j)_{1\le i<j\le n}$. Denote the Euler class
of~$\xi$ by $e\in H^2(X,\Z)$.

According to \cite[Lemma 2]{Milnor}, the result
of the standard pairing $\langle e, v_i\wedge v_j\rangle\in\Z$ can
be computed as follows. If the flat
bundle $\xi$ is defined by a homomorphism $\ph\colon
\Lambda\to \SL_2(\R)$, put $g_j=\ph(v_j)\in\PSL_2(\R)$, and pick, for
each~$j$, an element $\Gamma_j\in\uPSL=\widetilde{\SL_2(\R)}$ (the
universal covering) that 
is mapped to~$g_j$. Identifying the center of
$\widetilde{\SL_2(\R)}$ with~$\Z$, one has now
\begin{equation}\label{eq:Euler}
\langle e, v_i\wedge v_j\rangle=-[\Gamma_i,\Gamma_j].  
\end{equation}

 Since the group $\Lambda$ is
commutative, $[g_i,g_j]=1$, and Lemma~\ref{commutators} (applied to
the images of $g_1$ and~$g_2$ in~$\uPSL$) implies that
$[\Gamma_i,\Gamma_j]=1$ as well. Thus, \eqref{eq:Euler} implies that
$e=0$. We arrived at a contradiction, which completes the proof of
Propositions~\ref{on_torus} and~\ref{tori}. 
\end{proof}

\appendix

\section{Some remarks on flat line bundles}

The content of the appendix belongs to folklore and does not claim
novelty. Proofs will be omitted.

Throughout the section, we will be working with (topological) complex
line bundles on a fixed topological space~$X$, which is assumed to be
Hausdorff, path connected, locally simply connected, and paracompact
(any connected manifold has these properties, of course).

A complex line bundle on~$X$ is flat if and only if it admits a
triviaization with locally constant transition functions. We identify
the groups $\R/\Z$ and $\mathrm U(1)$ via the isomorphism $t\bmod
\Z\mapsto e^{2\pi it}$.

\begin{proposition}
Any flat line bundle on~$X$
is isomorphic to a line bundle that admits a trivialization
over an open cover of~$X$ such that the transition functions are
locally constant mappings to $\mathrm U(1)\subset\C^*$.
\end{proposition}

\begin{proof}[Sketch of proof]
Any such bundle is isomorphic to~$L_1\otimes L_2$, where the
transition functions of the line bundle~$L_1$ are locally constant of
modulus~$1$ and these of~$L_2$ are poisitive; it is clear that the
bundle~$L_2$ is trivial.
\end{proof}

\begin{corollary}\label{app:cor}
The following abelian groups are isomorphic.
  \begin{enumerate}
  \item
The group of flat complex line bundles.    
\item
  $\Hom(\pi_1(X),\mathrm U(1))$.
\item
  $\Hom(H_1(X,\Z),\mathrm U(1))$.
\item
$H^1(X,\mathrm U(1))$.  
  \end{enumerate}
\end{corollary}

Let $\delta_1\colon H^1(X,\R/\Z)\to H^2(X,\Z)$ be
the homomorphism from the long exact sequence induced by the exact
sequence
\begin{equation}\label{Z,R,R/Z}
0\to \Z\to\R\to\R/\Z\to 0  
\end{equation}
of abelian groups, and let
\begin{equation*}
\delta_2\colon \Hom(H_1(X,\Z),\R/\Z)\to \Ext^1(H_1(X,\Z),\Z)
\end{equation*}
be the homomorphism from the long exact sequence corresponding to $\Hom$
from $H_1(X,\Z)$ to the exact sequence~\eqref{Z,R,R/Z}.

Recall the exact sequence 
\begin{equation}\label{eq:uc}
0\to \Ext^1(H^1(X,\Z),\Z)\xrightarrow j H^2(X,\Z)\to \Hom(H_2(X,\Z),\Z)\to0,  
\end{equation}
the existence of which is asserted by the universal coefficients theorem.

\begin{proposition}\label{prop:CD}
  The following diagram is commutative:
\begin{equation}\label{CD}
    \xymatrix{
      {H^1(X,\R/\Z)}\ar[r]^{\delta_1}\ar@{=}[d]&{H^2(X,\Z)}\\
      {\Hom(H_1(X,\Z),\R/\Z)}\ar[r]^{\delta_2}&
    {\Ext^1(H^1(X,\Z),\Z)}\ar@{^{(}->}[u]_{j}  
    }      
\end{equation}
where $j$ is the injective
homomorphism from~\eqref{eq:uc}.
\end{proposition}

\begin{proposition}\label{c_1(/Ll)}
If $\xi\in H^1(X,\R/\Z)=\Hom(\pi_1(X),\R/\Z)$, then
$\delta_1(\xi)=c_1(L_\xi)$, where $L_\xi$ is the flat line bundle
corresponding to the homomorphism~$\xi$.   
\end{proposition}

Propositions~\ref{prop:CD} and~\ref{c_1(/Ll)} imply the following.

\begin{proposition}\label{4equiv}
For a complex line bundle $L$ on~$X$, the following assertions
are equivalent.
\begin{enumerate}
\item $L$ is flat.
\item $c_1(L)\in\Im(j)\subset H^2(X,\Z)$, where $j$ is the injection
  from the universal coefficients exact sequence~\eqref{eq:uc}.
\item $\langle c_1(L),x\rangle=0$ for any $x\in H_2(X,\Z)$, where 
  \begin{equation*}
\langle \cdot,\cdot\rangle\colon H^2(X,\Z)\otimes H_2(X,\Z)\to\Z    
  \end{equation*}
  is the standard pairing.
\item $c_\R(L)=0$.  
\end{enumerate}
\end{proposition}

Since, for any finiutely generated abelian group~$A$, $\Ext^1(A,\Z)$
is a torsion group, Proposition~\ref{4equiv} implies the following
fact, which is just Lemma~2.6 from~\cite{Putman}.

\begin{corollary}
If $H_1(X,\Z)$ is finitely generated \textup(in particular, if $X$ is
a compact manifold\textup), then a line bundle $L$ on~$X$ is
flat if and only if
$c_1(L)$ is a torsion element of~$H^2(X,\Z)$.  
\end{corollary}

\begin{remark}
There exist pairs $(X,L)$, where $X$ is a non-complact complex
manifold and $L$ is a flat holomorphic line bundle in~$X$, such that
$c_1(L)$ is not a torsion element of~$H^2(X,\Z)$. For example, let the
group~$G$ be the free product of groups $\Z/n\Z$ for all integer~$n>
1$. There exists a connected open subset $X\subset\C^n$ such that
$\pi_1(X)\cong G$ (actually, one may take~$n=3$). Since
$\Hom(G,\mathrm U(1))\cong \prod_{n>1}\Z/n\Z$, there exists a
homomorhism $\xi\colon\pi_1(X)\to\mathrm U(1)$ such that $\xi^k\ne1$
for any~$k\ge2$. If $L$ is the flat holomorphic line bundle on~$X$
corresponding to the homomorphism~$\xi$, then Corollary~\ref{app:cor}
implies that $L$ is not a torsion element in the group of
(topological) complex line bundles on~$X$, or, equivalently, that
$c_1(L)\notin H^2(X,\Z)_{\mathrm{tors}}$.
\end{remark}

\bibliographystyle{amsplain}

\bibliography{s2}

\providecommand{\bysame}{\leavevmode\hbox to3em{\hrulefill}\thinspace}
\providecommand{\MR}{\relax\ifhmode\unskip\space\fi MR }
\providecommand{\MRhref}[2]{%
  \href{http://www.ams.org/mathscinet-getitem?mr=#1}{#2}
}
\providecommand{\href}[2]{#2}
\begin{thebibliography}{10}

\bibitem{AV}
A.~Andreotti and E.~Vesentini, \emph{On the pseudo-rigidity of {S}tein
  manifolds}, Ann. Scuola Norm. Sup. Pisa Cl. Sci. (3) \textbf{16} (1962),
  213--223. \MR{170353}

\bibitem{Chirka}
E.~M. Chirka, \emph{Holomorphic motions and the uniformization of holomorphic
  families of {R}iemann surfaces}, Uspekhi Mat. Nauk \textbf{67} (2012),
  no.~6(408), 125--202 (Russian), English translation: Russian Mathematical
  Surveys, 2012, vol.~67, No.~6, 1091--1165. \MR{3075079}

\bibitem{FischerGrauert}
Wolfgang Fischer and Hans Grauert, \emph{Lokal-triviale {F}amilien kompakter
  komplexer {M}annigfaltigkeiten}, Nachr. Akad. Wiss. G\"{o}ttingen Math.-Phys.
  Kl. II \textbf{1965} (1965), 89--94. \MR{184258}

\bibitem{Grauert62}
Hans Grauert, \emph{\"{U}ber {M}odifikationen und exzeptionelle analytische
  {M}engen}, Math. Ann. \textbf{146} (1962), 331--368. \MR{137127}

\bibitem{Lehner}
Joseph Lehner, \emph{Discontinuous groups and automorphic functions},
  Mathematical Surveys, No. VIII, American Mathematical Society, Providence,
  RI, 1964. \MR{164033}

\bibitem{Milnor}
John Milnor, \emph{On the existence of a connection with curvature zero},
  Comment. Math. Helv. \textbf{32} (1958), 215--223. \MR{95518}

\bibitem{Narasimhan}
Mudumbai~S. Narasimhan, \emph{Variations of complex structures on an open
  {R}iemann surface}, Ann. Inst. Fourier (Grenoble) \textbf{11} (1961),
  493--514, XVI--XVII. \MR{125960}

\bibitem{NWW}
Karl-Hermann Neeb, Friedrich Wagemann, and Christoph Wockel, \emph{Making
  lifting obstructions explicit}, Proc. Lond. Math. Soc. (3) \textbf{106}
  (2013), no.~3, 589--620. \MR{3048551}

\bibitem{Nishino}
Toshio Nishino, \emph{Nouvelles recherches sur les fonctions enti\`eres de
  plusieurs variables complexes. {II}. {F}onctions enti\`eres qui se
  r\'{e}duisent \`a celles d'une variable}, J. Math. Kyoto Univ. \textbf{9}
  (1969), 221--274. \MR{255842}

\bibitem{Ohsawa}
Takeo Ohsawa, \emph{{$L^2$} proof of {N}ishino's rigidity theorem}, Kyoto J.
  Math. \textbf{60} (2020), no.~3, 1047--1050. \MR{4134358}

\bibitem{Putman}
Andrew Putman, \emph{The {P}icard group of the moduli space of curves with
  level structures}, Duke Math. J. \textbf{161} (2012), no.~4, 623--674.
  \MR{2891531}

\bibitem{Rawnsley}
John Rawnsley, \emph{On the universal covering group of the real symplectic
  group}, J. Geom. Phys. \textbf{62} (2012), no.~10, 2044--2058. \MR{2944792}

\bibitem{Siu}
Yum~Tong Siu, \emph{Techniques of extension of analytic objects}, Lecture Notes
  in Pure and Applied Mathematics, Vol. 8, Marcel Dekker, Inc., New York, 1974.
  \MR{361154}

\bibitem{Steenrod}
Norman Steenrod, \emph{The {T}opology of {F}ibre {B}undles}, Princeton
  Mathematical Series, vol. 14, Princeton University Press, Princeton, NJ,
  1951. \MR{39258}

\end{thebibliography}

\end{document}